\documentclass[12pt]{amsart}
\usepackage{amsmath,amsfonts,amssymb,amsthm}
\usepackage{framed}
\usepackage[dvipsnames]{color}
\usepackage[T1]{fontenc}
\usepackage{times} 
\usepackage[cp1250]{inputenc}  
\advance\topmargin by -\headheight
\advance\topmargin by -\headsep
\evensidemargin=\oddsidemargin
\makeatother
\newtheorem{theo}{Theorem}[section]
\newtheorem{prop}[theo]{Proposition}

\numberwithin{equation}{section}
\newcommand{\CA}{\mathcal{A}} 
\newcommand{\T}{\mathbb{T}} 
\newcommand{\oh}{{\tfrac{1}{2}}} 
\newcommand{\C}{\mathbb{C}} 
\newcommand{\sq}{\unskip\nobreak\kern5pt\nobreak\vrule height4pt width4pt depth0pt} 
\newbox\ncintdbox \newbox\ncinttbox
\setbox0=\hbox{$-$} \setbox2=\hbox{$\displaystyle\int$}
\setbox\ncintdbox=\hbox{\rlap{\hbox
    to \wd2{\kern-.1em\box2\relax\hfil}}\box0\kern.1em}
\setbox0=\hbox{$\vcenter{\hrule width 4pt}$}
\setbox2=\hbox{$\textstyle\int$}
\setbox\ncinttbox=\hbox{\rlap{\hbox
    to \wd2{\kern-.14em\box2\relax\hfil}}\box0\kern.1em}

\title{Gauss-Bonnet for matrix conformally rescaled Dirac}
\author[M.\ Khalkhali]{Masoud Khalkhali}
\address{Department of Mathematics, University of Western Ontario, London Ontario, N6A 5B7, Canada.}
\email{masoud@uwo.ca}
\author[A.\ Sitarz]{Andrzej Sitarz}
\address{Institute of Physics, Jagiellonian University,
	prof.\ Stanis\l awa \L ojasiewicza 11, 30-348 Krak\'ow, Poland.\newline\indent
	Institute of Mathematics of the Polish Academy of Sciences,
	\'Sniadeckich 8, 00-950 Warszawa, Poland.}
\email{andrzej.sitarz@uj.edu.pl}   
\begin{document}
\maketitle
\begin{abstract}
We derive an explicit formula for the scalar curvature over a two-torus with a Dirac operator
conformally rescaled by a globally diagonalizable matrix. We show that the Gauss-Bonnet theorem
holds and extend the result to all Riemann surfaces with Dirac operators modified in the same 
way.  
\end{abstract}
\thispagestyle{empty}

\section{Introduction}

In this paper we consider a new class of noncommutative algebras that  lend themselves to spectral and geometric  analysis. In noncommutative geometry the metric 
structure on a noncommutative space $\CA$ is encoded in a spectral triple on $\CA$ and under suitable conditions, using the spectral data,   one can formulate a notion of scalar and Ricci curvature for $\CA$ \cite{CT11, FGK}. In recent years much progress has been made in understanding and computing  the  curvature invariants of   curved  noncommutative tori using spectral properties of the Dirac operator \cite{CM11,CT11,DS13,FK15,FGK14,FK13,FK12, FGK, F15}.
In this paper we consider a different class of noncommutative algebras, namely algebras of matrix valued  functions on a classical 2 dimensional Riemannian manifold and prove a Gauss-Bonnet theorem for them.  

Let $M$ be a two dimensional Riemannian manifold which we assume to be closed, connected and oriented with a fixed spin structure.   Let  $D: L^2(S) \to L^2(S)$  denote the Dirac operator of $M$ acting on the space  
of spinors.    
Consider the algebra $\CA = C^\infty(M) \otimes M_n(\mathbb{C})$ of smooth matrix valued  functions on $M$ and its diagonal representation on $\mathcal{H} = L^2(S) \otimes \mathbb{C}^n$.  Taking the diagonal action of the Dirac operator we obtain a spectral triple on $\CA$.  Let $h \in \CA$ be a positive
element. We use $h$ to perturb the spectral triple of $\CA$
in the following way. Consider the operator $D_h = h D h$ as 
a conformally rescaled Dirac operator. Although $D_h$ does not
have bounded commutators with the elements of $\CA$ we might 
still consider the resulting geometry as a twisted spectral triple \cite{FO16}
or pass to $h$ in the commutant of $\CA$ \cite{BCDS}. In either case the perturbation
is of the same type.  We would like to address the question:  does the Gauss-Bonnet theorem hold  for $D_h$?  We provide a positive answer to this question. 

Recall that the Gauss-Bonnet theorem  for two dimensional manifolds can be stated in terms of spectral zeta functions.  For $n=1$, that is $\CA = C^\infty(M),$ where one deals with a commutative algebra,  the spectral  zeta function $\zeta_D (s)$ associated to the Dirac  Laplacian $D^* D$  is defined  by 
\begin{equation*} 
\zeta_D (s) = \sum \lambda_{j}^{-s}, \,\,\, \quad \textnormal{Re}(s) > 1, 
\end{equation*}
where the  summation is over non-zero eigenvalues of $D^*D$. 
The zeta function is absolutely convergent and holomorphic for  $\textnormal{Re}(s) > 1$ and has a meromorphic continuation to $\mathbb{C}$ with a unique (simple) pole at $s=1.$  In particular 
$\zeta (0)$ is defined and it is well know that it  is a topological 
invariant. For example,  in the special case  where  $D=\bar{\partial}$ 
is the Cauchy-Riemann operator of the complex structure defined by the conformal class of the metric, we have  
\[ \zeta_D(0) + 1 = \frac{1}{12 \pi} \int_{M} R = \frac{1}{6}
\chi (M), \nonumber \]
where $R$ is the scalar curvature and $\chi (M)$ is the Euler-Poincar\'e characteristic. Thus $\zeta_D (0)$ is a topological invariant, and,  in particular, it remains invariant under  the conformal perturbation $g \to e^fg$ of the metric \cite{CM11}. 

For technical reasons, we  assume there is a unitary element $U \in \CA$ such 
that $ h = U H U^*, $ where  $H$ is a diagonal matrix. Then:
$$ h D h = U H U^* D U H U^* = U \left( H \left( D + U^* [D,U] \right) H \right) U^*. $$
Therefore the spectrum of $D_h$ is the same as the  spectrum of $D_{A,H} = H (D + A) H $, where $A$ is a matrix valued one-form. 

In this  paper we  show that the Gauss-Bonnet theorem  holds for the family of conformally rescaled Dirac operators with possible fluctuations 
$D_{A,H} = H (D + A) H $ where the rescaling is a diagonalizable matrix 
and we compute the local expressions for the scalar curvature. The main point
of using the diagonal matrix $H$ as the conformal rescaling is that $H$ 
commutes then with all its derivatives thus making the computations feasible.
 
In the computations we use the matrix valued pseudodifferential 
operators over the manifold as opposed to the methods used in 
\cite{IM17}. The results demonstrate that unlike in the case of 
higher residues there, the expressions for the value of the 
$\zeta$ function at $0$ are complicated also in the matrix 
case.

Among other examples of curvature computations for  noncommutative manifolds  we should mention 
the Moyal sphere \cite{ESW16} where the algebraic approach of \cite{Ro13}
was applied,  and the differential geometry based construction for the 4-sphere
\cite{AW17},   as well as the toric noncommutative manifolds \cite{Li15}.

\section{Computations for the torus}

Consider the canonical spectral triple for a flat two torus 
$M= \mathbb{R}^2/\mathbb{Z}^2.$  Its spin structure is defined
by the Pauli spin matrices $\sigma^1, \sigma^2$ and its  Dirac operator is 
$$ D = \sigma^1 \delta_1 + \sigma^2 \delta_2. $$
Here $\delta_1, \delta_2$ are the partial derivatives $\frac{1}{i}\frac{\partial}{\partial x}$  and  $\frac{1}{i}\frac{\partial}{\partial y}.$ 

To compute the resolvent kernel we work in the algebra of matrix valued pseudodifferential operators obtained by tensoring the algebra  ${\bf \Psi}$   of pseudodifferential operators  on a smooth manifold $M$ by the algebra of $n$ by $n$ matrices. Let  $ h = U H U^* $  be a conformal factor 
where  $H$ is a diagonal matrix-valued function and $U$ a unitary
matrix-valued function on a torus.

\subsection{Computing the resolvent}

We have the following form of the symbol of the operator
$D_{A,H}^2 = H (D+A) H^2 (D+A) H, $ 
$$ \sigma_{D_{A,H}^2} = a_2 + a_1 + a_0, $$
where:
\begin{equation}
\begin{aligned}
a_2 =& H^4 \xi^2, \\
a_1 =& i \epsilon_{ij} \sigma^3 2 H^3 \delta_i(H) \xi^j 
      + 4 H^3 \delta_i(H) \xi^i 
      - i \epsilon_{ij} \sigma^3 H^3 A_i H \xi^j \\ 
     & + H^3 A_i H \xi^i 
      + i \epsilon_{ij} \sigma^3 H A_i H^3 \xi^j 
      + H A_i H^3 \xi^i, \\ 
a_0 =&   H^4 (\Delta H)
       + H^3 A_j \delta_i(H)   
       - H^3 i \sigma^3 \epsilon_{ij} A_i \delta_j(H)
       + H^3 \delta_i(A_i) H
       + i \sigma^3 H^3 \epsilon_{ij} \delta_j(A_i) H \\
     & + 2 H^2 \delta_i(H) \delta_i(H)  
       + 2 H^2 \delta_i(H) A_i H
       + 2 H A_i H^2 \delta_i(H) 
       + 2 i \sigma^3 H^2 \epsilon_{ij} \delta_i(H) A_j H \\
     & + i \sigma^3 \epsilon_{ij} H A_i H^2 \delta_i(H)  
       + i \sigma^3 \epsilon_{ij} H A_i H^2 A_j 
       + H A_i H^2 A_i H.         
\end{aligned}
\label{bsym}
\end{equation} 

Then the first symbols of $(D_{H,a})^{-2} = b_0 + b_1 +b_2 + \cdots$
are:

\begin{equation}
\begin{aligned}
&b_0  = (a_2+1)^{-1}, \\
&b_1  = -  \left( b_0 a_1 +  \partial_k(b_0) \delta_k(a_2) \right) b_0, \\
&b_2  = - \left(  b_1 a_1 + b_0 a_0 +  \partial_k(b_0) \delta_k(a_1) 
+  \partial_k(b_1) \delta_k(a_2) + \oh \partial_k \partial_j (b_0) \delta_k \delta_j (a_2) \right) b_0.
\end{aligned}
\label{bsym2}
\end{equation} 

Here:
$$ b_0 = (1 + H^4 \xi^2)^{-1}, $$
and in the above formulas we have used that $H$ commutes with $\delta_i(H)$.

The computations of $b_2$ yield three parts, which we treat
separately: independent of $A$, linear in $A$, quadratic in $A$
and terms depending on the derivate of $A$.

\section{Curvature}

Apart from the computations of the Gauss-Bonnet term, it is interesting to
obtain an explicit formula for the local curvature term (or more accurately, 
for the curvature modified by the local measure). We use the noncommutative geometry setup to compute  the scalar curvature functional in the sense defined below.

A matrix-valued function $R: \T^2 \to M_n(\C)$we call a local curvature 
functional if for any $f$ from the algebra of matrix valued functions on 
the torus:
$$ \zeta_{f,D}(0) = \int_{\T^2} \hbox{Tr} \, f R, $$
where
$$ \zeta_{f,D} = \hbox{Tr} \; f |D|^{-s}.$$
Note that since we are working with the operator that is defined on the flat
torus, the integral is with respect to the standard flat metric, hence the
matrix valued function $R$ contains both the scalar curvature as well as 
the volume form arising from the metric defined by the Dirac operator 
$D$.

Using our assumptions and notation, for the Dirac operator, which has been 
conformally modified by a globally diagonalizable function, we obtain four 
contributions to the curvature functional term. In all computations below we
use the fact that $H$, $\delta_i(H)$ and $b_0$ commute with each other.

\subsection{Terms not depending on $A$}

First term, which is idependent of $A$ and contains only functions of $H$
and its derivatives $\delta_i(H)$, is therefore identical with the classical 
contribution ,
\begin{equation}
\begin{aligned}
b_2(H, \xi) = & 96\, b_0^5 \delta_i(H)\delta_i(H) H^{14} (\xi^2)^3 
                        - 136\,  b_0^4 \delta_i(H)\delta_i(H) H^{10} (\xi^2)^2 \\
& + 46\, b_0^3 \delta_i(H)\delta_i(H) H^{6} (\xi^2)
- 2\, b_0^2 \delta_i(H)\delta_i(H) H^{2}  \\
& - 8 \, b_0^4 \Delta(H) H^{11} (\xi^2)^2
  + 8 \, b_0^3 \Delta(H) H^{7} (\xi^2) 
  - \, b_0^2 \Delta(H) H^3,
\end{aligned}
\end{equation} 

Integrating out over the $\xi$ space and using
$$ \int_0^\infty \frac{r^{2k+1} dr}{  (1+a^2 r^2)^{2k+3} } = \frac{1}{2(k+1)a^{2(k+1)}} $$ 

we obtain:

\begin{equation}
R(H) = - \pi \left( \frac{1}{3} H^{-2} \delta_i(H) \delta_i(H) + \frac{1}{3} H^{-1} \Delta(H) \right).
\label{RH}
\end{equation}

In the subsequent computation we use Lesch rearrangement lemma
\cite{Le14} and denote the conjugation by $\Delta$ conjugation by 
$H^4$:
$$ \Delta(x) = H^{-4} x H^4.$$

\subsection{Terms linear in $A$}

We have,
$$
\begin{aligned}
b_2^{(1)}(H,A) = & 
    - b_0 h A_i  b_0 \delta_i(H)  H^2 
+ 5 b_0 h A_i b_0^2 \delta_i(H) H^6 \xi^2 
- 4 b_0 h A_i b_0^3 \delta_i(H) H^{10} (\xi^2)^2 \\
&   - b_0 H^3 A_i b_0 \delta_i(H)
+ 7 b_0 H^3 A_i b_0^2 \delta_i(H) H^4 \xi^2 
- 4 b_0 H^3 A_i b_0^3 \delta_i(H) H^8 (\xi^2)^2 \\
& + 3 b_0^2 H^5 A_i b_0 \delta_i(H)  H^2 \xi^2
 - 4 b_0^2 H^5 A_i b_0^2 \delta_i(H)  H^6 (\xi^2)^2 \\
& + b_0^2 H^7 A_i b_0 \delta_i(H) \xi^2 
- 4  b_0^2 H^7 A_i b_0^2 \delta_i(H) H^4 (\xi^2)^2 \\ 
\end{aligned}
$$
and
$$
\begin{aligned}
b_2^{(1)}(H,A) &= - 2 b_0 \delta_i(H)  H^2 A_i b_0 H 
    + 2 b_0^2 \delta_i(H)  H^4 A_i b_0 H^3 \xi^2  \\
& + 6 b_0^2 \delta_i(H)  H^6 A_i b_0 H \xi^2
   -  4 b_0^3 \delta_i(H)  H^8 A_i b_0 H^3 (\xi^2)^2 \\
&-  4 b_0^3 \delta_i(H)  H^{10} A_i b_0 H (\xi^2)^2.
\end{aligned}
$$
Explicit computations give first:
$$ R^{(1)}(H,A) =  \sum_{i=1,2} 2 \pi H G(\Delta)(A_i) \delta_i(H), $$
where $G$ is the following function:
$$ G(s) = 
{ (1+\sqrt{s})\sqrt{s} \over (s-1)^3 } 
\big( (s+1) \ln (s)  - 2 (s-1) \big), $$
and a second term,
$$ R^{(2)}(H,A) =  \sum_{i=1,2} -2 \pi H^{-2} \delta_i(H) G(\Delta)(A_i) H, $$
surprisingly with the same function $G(s)$.

Note that after taking the trace both terms shall cancel each other independently
of its value at $s=1$, which is $ G(1) = \frac{2}{3}$.

Therefore
$$ \hbox{Tr\ } \left( R^{(1)}(H,A)  + R^{(2)}(H,A)  \right) = 0. $$

\subsection{Terms linear in $\delta_i(A_i)$}

In this case we have:

$$
\begin{aligned}
b_2(H,\delta_i(A_j)) & =
   - b_0 H^3 \delta_i(A_i) b_0 H
   + b_0^2 H^5 \delta_i(A_i) b_0 H^3 \xi^2 \\
&+ b_0^2 H^7 \delta_i(A_i) b_0 H \xi^2,
\end{aligned}
$$

and again explicit integration over $\xi$ gives:
$$ \pi H^{-1} F(\Delta) (\delta_i(A_i)) H,$$
where
$$F = - \frac{(1+\sqrt{s}) \sqrt{s}}{(s-1)^2} \ln(s) + \frac{\sqrt{s}+1}{s-1}. $$
Again, it is not difficult to check that $F(1)=0$ and the expression vanishes
after we take the trace of it, so:

$$ \hbox{Tr\ } \left( R(H,\delta_i(A_j)) \right) = 0. $$

\subsection{Quadratic terms in $A_i$.}

We have:

$$ 
\begin{aligned}
b_2(H,A^2) &= - b_0 H A_i H^2 A_i b_0 H 
    + b_0 H A_i b_0 H^6 A_i b_0 H \xi^2 \\
& + b_0 H^3 A_i b_0 H^2 A_i b_0 H^3 \xi^2 .
\end{aligned}
$$

Integrating over $\xi$ we obtain:

$$
R(H,A^2) = - \pi H^{-1} Q(\Delta^{(1)},\Delta^{(2)})(A_i \cdot A_i) H
$$
where
$$
Q(s,t) = { \sqrt{s} (\sqrt{t} + s) \over (s-1) )(s-t) } \; \ln s 
- { \sqrt{s} \sqrt{s} \over (s-t) \sqrt{t} } \; \ln t 
$$

To compute the trace we first take $t=1$:
$$ F(s) = Q(s,1) = { (s+1) \ln s + 2 (1-s) \over  (s-1)^2 }, $$
and observe that due to the trace property:
$$ 
\begin{aligned}
\hbox{Tr} \left( H^{-1} F(\Delta)(A_i) A_i H \right) &= 
\hbox{Tr} \left( A_i  F(\Delta)(A_i)  \right) \\
& = \hbox{Tr} \left(  F(\Delta^{-1})(A_i) A_i  \right). 
 \end{aligned}
$$ 
But:
$$ 
\begin{aligned}
F(\frac{1}{s})
&= { - (\frac{1}{s}+1) \ln s + 2 (1-\frac{1}{s}) \over  (\frac{1}{s}-1)^2 } \\
&= { - (s+1) \ln s - 2 (1-s) \over (s-1)^2} \\
&= - F(s),
\end{aligned} 
$$
and therefore 
$$  \hbox{Tr\ } R(H,A^2) = 0, $$
so the quadratic term vanishes as well.

\section{The Gauss-Bonnet theorem}

The term which does not depend on $A$  is a total derivative term:

$$ \frac{1}{3} \delta_i \left( H^{-1} \delta_i(H) \right) $$

Since the trace is closed on the torus this contribution to the Gauss-Bonnet term vanishes. For the linear terms as well as for the quadratic we have already demonstrated that taking the trace gives $0$, similarly as for the term depending on the derivative of  $A$. Hence the conclusion:

\begin{prop}
For the matrix conformally rescaled Dirac operator on the 
two-dimensional torus,
$D_h = h D h$, where $h$ is a globally diagonalizable positive matrix,
the Gauss-Bonnet theorem holds:
$$ \zeta_{D_h}(0) = \zeta_{D}(0). $$
\end{prop}

\section{Matrix Gauss-Bonnet for an arbitrary two-dimensional manifold}

Let $M$ be a closed, connected,  two-dimensional Riemannian manifold 
and $D$ a Dirac operator for a fixed metric $g$ on $M$.  
Consider the operator
$$ D_{H,A} = H(D + A)H, $$
for $H$ a diagonal matrix valued function on $M$ and $A$ a matrix-valued one-form (identified here with their Clifford image).

We again aim to compute the value of $\zeta_{D_{H,A}^2}(0)$
using again the methods of pseudodifferential calculus. Of course,
using the pseudodifferential calculus on a compact Riemannian
manifold requires some care, as the formulation depends on local
coordinates and requires patching together local data (see \cite{NSSS}
for a recent review and literature) using partition of unity.

In particular, for a curved manifold the expressions for the product of
the symbols that were used in (\ref{bsym2}) become
complicated (even using local charts and local coordinates)
as the metric tensor depends on them. An example of the complexity
for the products in the pseudodifferential calculus that uses normal
symbols (based on the normal coordinates) is given in \cite{MP}.
In our case, however, we are interested only in the contributions that
contain the term $A$ and therefore we can easily use the local
arguments. Moreover, due to seminal results of Guillemin and
Wodzicki \cite{Gu,Wo}, the trace of the integral over the cotangent space of a
classical symbol of order $-2$ provides a local density
on a $2$-dimensional manifold and thus our computations
are indeed coordinate independent.

Let us denote the (local) symbols of $D^2_H$ as:
$$ D_H^2 = (HDH)^2 = a_2^H + a_1^H + a_0^H, $$
and the symbols of $D^2$ alone as:
$$ D^2 = a_2^o + a_1^o +a_0^o. $$

Similarly, like in the case of the torus we split the computation
into the case of terms not depending on $A$, linear in $A$
and quadratic in $A$.

\subsection{Terms independent of $A$}

As the matrix $H$ is diagonal, we can treat the case as $H$ were 
a scalar function. Thus the problem is reduced to the problem of a 
usual conformal rescaling of the classical Dirac operator. 

It is well known that for any conformal rescaling of the metric the
Gauss-Bonnet theorem holds, therefore the contribution of the
part of $b_2$ that does not contain $A$ guarantees that the value
of the zeta function will remain unchanged provided that all 
contributions depending on $A$ shall vanish. We examine the
linear and quadratic contributions at each point $x \!\in\! M$ of the 
manifolds using  normal coordinates.

\subsection{Terms linear in $A$}

Linear terms do arise in $b_2$ from the following terms (using 
local coordinates in a given chart):
$$ 
\begin{aligned}
&b_0 a_1^H b_0 a_1(A) b_0  
+ \partial_k^\xi(b_0) \partial^x_k(a_2^H) b_0 a_1(A) b_0
+ b_0 a_1(A) b_0 a_1^H b_0 \\
&- b_0 a_0(A) b_0 
- \partial_k^\xi(b_0) \partial^x_k(a_1(A)) b_0 
- \partial_k^\xi(b_0 a_1(A) b_0) \partial^x_k(a_2^H) b_0. 
\end{aligned}
$$  
where by $a_1(A), a_0(A)$ we denote the terms containing linear
$A$ in $a(D_{H,A}^2)$, respectively. 

First of all, observe that the expression in much simpler as it involves
only (at most) first-order derivatives. As to compute the density we 
can use any coordinate system let us choose the normal coordinates 
around $x \!\in\! M$ with respect to the metric $g$. First, the terms 
without derivatives reduce 
easily to the torus case (at point $x$). The only difficulty arises
from terms with derivatives with respect to normal coordinates,
that is, $ \partial^x_k(a_2^H)$ and $\partial^x_k(a_1(A))$.

However, since $a_2^H = H^4 g_{ij} \xi^i \xi^j$, then we use the fact
that normal coordinates the first derivatives of the metric vanish at 
the point $x$ and therefore the only remaining term would be the 
derivative of $H^4$.  Therefore, the term that contains the derivative
of $a_2^H$ would be reduced to the term linear in $A$ from the torus 
case (a point $x$ and with the derivatives taken with respect to the
normal coordinates).

Similar argument works also for the other term, $a_1(A)$, which reads:
$$ \left( H^3 A_i H + H A_k H^3 \right) \sigma^k \sigma^i \xi_i, $$
and because in the term $\partial_k^\xi(b_0) \partial^x_k(a_1(A)) b_0 $
there are no further $\sigma$ matrices we can compute first the 
trace over the Clifford algebra and rephrase is as
$$\frac{1}{2} \left( H^3 A_i H + H A_k H^3 \right) g^{ki} \xi_i, $$
obtaining again the metric. Hence, in normal coordinates at point $x$ 
the expression is again identical (in the sense of the dependence on
$A$ and $H$) to the one for the flat torus. 

\subsection{Quadratic terms}

Now, let us concentrate on quadratic terms in $A$ in
the formula for $b_2$. They can arise only from 
two terms and are:

$$ b_0 a_1(A) b_0 a_1(A) b_0 - b_0 a_0(A^2) b_0, $$

where $a_1(A)$  denotes the part of this symbol $a_1$
which contains a term linear in $A$ and $a_0(A^2)$
similarly denotes part of $a_0$ symbol containing the
quadratic term. Note that this case is even simpler as there
are no derivates and  it is easy to see that when
written in normal coordinates the expression becomes:

$$ a_1(A) = \left( H^3 \sigma^j \xi_j (\sigma^i A_i) H 
             + \sigma^i H A_i H^3 \sigma^j \xi_j \right), $$
and
$$ a_0(A) = (\sigma^i H A_i H) (\sigma^k H A_k H), $$
which again is naturally the same as in the case of torus.

\subsection{The local density and the Gauss-Bonnet}

As we have pointed out earlier, the symbol that we compute
is of order $-2$ and integrated over the variables $\xi$ contributes 
to a local density, so we can compute the value at each point in
any coordinate system. Using the normal coordinate system we
have shown that at each point $x$ using the local normal coordinate
systems  the dependence on the $A$ term is given through 
identical expressions as in the case of the flat torus. Therefore,
using the arguments from previous section, we see that
at each point $x \!\in\! M$ the linear and quadratic terms in $A$ 
give no contributions to the local density and hence to the 
value of the zeta function of $D_H^2$. We note in passing that
the use of normal coordinates in similar problems related to
the computations of the density of Wodzicki residue was used,
for example, in \cite{KW}.

On the other hand, since the matrix-valued function $H$ is assumed 
to be diagonal, the $H$ dependent terms can be treated in the
same way as the scalar modification of the Dirac operator, that is,
a case where $H$ is just a function on $M$. The latter case, however,
obviously does not change the Gauss-Bonnet term, hence as 
a consequence, conbining these two results we conclude that the
Gauss-Bonnet theorem holds for the arbitrary conformal rescaling
of the Dirac operator over a two-dimensional Riemann surface if
the rescaling matrix is diagonalizable. We shall briefly discuss in the
next section when such situation is possible.

\section{Diagonalizability of matrix functions}

Having demonstrated that for the special case of globally 
diagonalizable matrix the Gauss-Bonnet theorem holds, 
a natural question arises as to what extent the diagonalizability 
condition  $h= UHU^*$ is general. Is it always possible to find the unitary $U$ ? In what follows we analyze the question in more detail.

Let $X$ be a compact Hausdorff space and let $H: X \to M_n (\mathbb{C})$ be a continuous  map  with values in positive definite matrices. We also assume that for all $x \in X$, $H (x)$ has simple spectrum. A natural question is  if $H$ can be continuously diagonalized. That is,  if there is a decomposition $h = UHU^*$ with  $h$ diagonal, $U$ unitary,  and both  
$h$ and $U$ continuous.  A similar question has been studied for normal 
and selfadjoint maps in \cite{FP} where obstructions to continuous diagonalizability 
are identified. In our case of positive definite matrices, these obstructions are much 
easier to identify and the necessity of their vanishing are directly proved in this section. 

We give now the complete obstruction for diagonalizability of $H$ in terms of 
first Chern classes. Given a map $H: X \to M_n (\mathbb{C})$ as above, let  
$\lambda_1 (x) < \lambda_2(x) <\dots < \lambda_n(x), x \in X$ denote 
the  eigenvalues of $H(x)$. Since $H$ is continuous and has a simple real  spectrum,  
$\lambda_i'$s are continuous functions on $X$. Let $E_i(x)$
denote the corresponding eigenspaces. We obtain complex line bundles $L_i$ on $X$,
$$ L_i \subset X \times \mathbb{C}^n= \{ (x, v); v \in E_i(x)\}. $$  
It is clear that $H$ is continuously diagonalizable if and only if the line bundles $L_i$ 
are trivial for all $i$. The obstruction for triviality of complex  line bundle $L$  
is therefore given by its  first Chern class: $L$ is trivial if and only if $c_1 (L)=0$. Thus 
the obstruction to diagonalizability of $H$ lies in 
$\oplus_{i=1}^n H^2 (X, \mathbb{Z})$ so that $H$ is diagonalizable if and only if
$$ c_1(L_i) =0, \quad \quad 1\leq i \leq n.$$

In particular we get the following:
\begin{prop} Let $X$ be a compact connected  Hausdorff space  such that 
$H^2 (X, \mathbb{Z})=0.$ 
Then any continuous map $H: X \to M_n (\mathbb{C})$  with values in positive definite 
matrices with simple spectrum is continuously diagonalizable. 
\end{prop}

It should be noted that if $X$ is a smooth manifold and $H$ is smooth as well, then the roots 
$\lambda_i(x)$ will be smooth functions on  $X$. It follows that the line bundles $L_i$ will be smooth and in that case $H$ will be smoothly diagonalizable, provided of course 
$H^2(X, \mathbb{Z})=0.$

On the other hand it is not difficult to give examples to show that continuous 
diagonalizability is not always possible. The simplest example is a sphere with the 
function
$$ F: S^2 \rightarrow M_2(\mathbb{C}), 
\quad F(x_1, x_2, x_3)=x_1\sigma_1 +x_2\sigma_2 +x_3\sigma_3,$$ 
where $\sigma_1, \sigma_2, \sigma_3$ are the Pauli spin matrices. They are selfadjoint and satisfy the relations $\sigma_i \sigma_j + \sigma_j \sigma_i = 2 \delta_{ij}. $ 
Then $F^2 (x)=1$, the identity matrix,  for all $x\in S^2$ and
therefore
$p=\frac{1+F}{2}$ is a
projection  in $M_2(C(S^2))$.  We have
$$ p(x_1, x_2, x_3)= \frac{1}{2} \left(
\begin{matrix}
1 + x_3 & x_1 + ix_2 \\
x_1-ix_2 & 1-x_3
\end{matrix}
\right).
$$
Thus $p$ has rank 1 and hence defines a complex line bundle over $S^2$. This line bundle is non-trivial in the sense that it admits no nowhere zero section, hence $p$  cannot be continuously diagonalized.  In fact,   it  can be shown that this  line bundle  is the line bundle associated to the Hopf fibration
$$ S^1 \to S^3 \to S^2,$$
  which  has no nowhere zero section hence cannot be a trivial line bundle.   Alternatively, the first  Chern class of this bundle can be explicitly computed as
$$ c_1 (p)= \frac{1}{2\pi i} \text{Tr}(p\, dp\, dp) = 
-\frac{1}{4 \pi} (x_1 dx_2 dx_3 -x_2dx_1 dx_3 +x_3 dx_1 dx_2),$$ 
which is  a multiple of the volume form of the round sphere. In particular $\int_{S^2} c_1(p)=-1$.  Now to get a non-diagonalizable positive definite map, let 
$H = 1+p.$  It has  a simple spectrum for all $x \in S^2$ and is positive definite. Since $p$ is not continuously diagonalizable, it follows that $H$ is not continuously diagonalizable either. 

In fact, one can extend the result to all compact Riemann surfaces.

\begin{prop}
If $X$ is a compact Riemann surface then there exists a non-diagonalizable
smooth positive $2 \times 2$ matrix. 
\end{prop}

\begin{proof}
To show that we shall find a hermitian projection $p$ that has a nontrivial Chern class and therefore cannot be diagonalized globally. Then, similarly as above, $1+p$ is the desired positive non-diagonalizable matrix. The
construction below extends the method of finding an explicit nontrivial
projection on the torus as shown in \cite{Lo86}.

Let us take a product $S^1 \times (0,1)$ and define the following matrix valued
function:
$$ p = \left( \begin{array}{cc}
f(t) & h(t) + g(t) e^{2\pi i s} \\
h(t) + g(t) e^{-2\pi i s}  &1- f(t) \end{array} \right),
$$
where $0 < t < 1$ parametrizes the interval and $0 \leq s <1 $ parametrizes the
circle. The matrix $p(t,s)$ is a projection iff:
$$ g(t) h(t) = 0, \qquad f(t)^2 + g(t)^2 +h(t)^2 = f(t). $$ 

The first Chern class of $p$ could be explicitly computed as:
$$ 
\begin{aligned}
c_1 (p) &= \frac{1}{2\pi i} \text{Tr}(p\, dp\, dp) = (4 g g' f - 4 g^2 f' -2 g g') \, dt \wedge ds \\
& = \left( 2 (g^2)' f - 4 g^2 f' - (g^2)' \right)  \, dt \wedge ds .
\end{aligned}
$$

Now, choose the functions $g,h$ in such a way, so that support of $g$ is 
in $(\epsilon, \frac{1}{2})$ and support of $h$ in $(\frac{1}{2}, 1-\epsilon)$ and they vary 
between $0$ and $1$. Further, take $f$ to be increasing on support of 
$g$ from $0$ to $1$ and decreasing on support of $h$, we can check that 
the integral of the above form:
$$ 
\int_{\hbox{supp}(g)} \left( 2 (g^2)' f - 4 g^2 f' - (g^2)' \right)  \, dt \wedge ds = -1,
$$
so that the projection $p$ is nontrivial.

Take now an arbitrary Riemann surface and find a circle and its tubular
neighborhood. Then, as this tubular neighborhood is diffeomorphic with $(0,1) \times S^1$
we can use that diffeomorphism to define a matrix value map on it, which arises
from the projection $p$. However, note that the projection $p(t,s)$ is constant:
$$ p = \left( \begin{array}{cc}
0 & 0 \\
0 & 1
\end{array} \right), $$
for $t< \epsilon$ and $t> 1 - \epsilon$, and so is its image. Therefore we can extend it in a smooth way to a matrix valued map on the entire Riemann surface. As the Chern number
of the projection does not change it defines a nontrivial line bundle. Taking $1+p$
we obtain a positive matrix (with constant eigenvalues $1$ and $2$), so that both
eigenspaces are nontrivial line bundles.
\end{proof}

The above result does not generalize in a natural way to higher dimensional manifolds, and we can formulate a question as follows. 
For a compact manifold $X$ of dimension bigger than $2$, if $H^2 (X, \mathbb{Z})$ is non-trivial, can one always find a positive matrix-valued $H$, which is not diagonalizable?
If so, what is the minimal size of such a matrix?
Another interesting problem is of different type: 
having a smooth matrix valued function $H$ on a manifold $X$, is there
a Chern-Weil type description of obstructions to diagonalizability (which are obviously related
to classes $c_1 (L_i)$) in terms of differential forms obtained from $H$?

Finally, let us mention that although for the 2-dimensional manifolds the action remains purely topological, the study of the matrix-conformal 
rescaling in the higher-dimensional case leads to a question of the minima of the corresponding Einstein-Hilbert functional and the corresponding 
equations of motion. 

\vfill

{\bf Acknowledgements:} The authors thank the referee for careful of the manuscript reading and remarks that improved the presentation. The research was partially supported through  H2020-MSCA-RISE-2015-691246-QUANTUM DYNAMICS and through Polish support	grant for the international cooperation project 3542/H2020/2016/2 and 328941/PnH/2016.

\end{document}